\documentclass[11pt,a4paper]{amsart} 

\usepackage{latexsym}
\usepackage{amsmath}
\usepackage{amsthm}
\usepackage{latexsym}
\usepackage{amssymb}
\usepackage[a4paper , lmargin = {3cm} , rmargin = {3cm} , tmargin = {2.5cm} , bmargin = {2.5cm} ]{geometry}

\usepackage{hyperref}

\theoremstyle{plain}
\newtheorem*{NewTheoremA}{Theorem A}
\newtheorem*{NewPropositionB}{Proposition B}
\newtheorem*{NewCorollaryC}{Corollary C}
\newtheorem*{NewTheoremD}{Theorem D}
\newtheorem*{NewPropositionE}{Proposition E}

\newtheorem{theorem}{Theorem}[section]
\newtheorem{lemma}[theorem]{Lemma}

\newtheorem{corollary}[theorem]{Corollary}

\newtheorem{proposition}[theorem]{Proposition}
\newtheorem{definition}[theorem]{Definition}

\title[]{on word hyperbolicity and virtual freeness of automorphism groups}
\author{Olga Varghese}
\thanks{Funded by the Deutsche 
Forschungsgemeinschaft (DFG, German Research Foundation) under Germany's 
Excellence Strategy -EXC 2044-, Mathematics M\"unster: Dynamics-Geometry-Structure}
\date{\today}

\address{Olga Varghese\\
Department of Mathematics\\
M\"unster University\\ 
Einsteinstra\ss e 62\\
48149 M\"unster (Germany)}
\email{olga.varghese@uni-muenster.de}

\begin{document}
\pagenumbering{arabic}
\begin{abstract}
We show that word hyperbolicity of automorphism groups of graph products $G_\Gamma$ and of Coxeter groups $W_\Gamma$ depends strongly on the shape of the defining graph $\Gamma$. We also characterized those ${\rm Aut}(G_\Gamma)$ and ${\rm Aut}(W_\Gamma)$ in terms of $\Gamma$ that are virtually free. 
\end{abstract}
\maketitle

\section{Introduction}
In this article we study automorphism groups of graph products and Coxeter groups focusing on two powerful properties: (i) word hyperbolicity in the sense of Gromov \cite{Gromov} and (ii) virtual freeness.
 
Given a finite simplicial graph $\Gamma=(V, E)$ and for each vertex $v\in V$ a non-trivial group $G_v$, the graph product $G_\Gamma$ is the free product of the vertex groups with added relations that imply elements of adjacent vertex groups commute. These groups were introduced by Baudisch in \cite{Baudisch} for infinite cyclic vertex groups and later by Green in \cite{Green} for arbitrary vertex groups. Word hyperbolic graph products of finite groups were characterized by Meier in \cite{Meier}. He showed, that $G_\Gamma$ is word hyperbolic iff $\Gamma$ contains no induced cycle of length four. In this paper we study word hyperbolicity in the setting of automorphism groups of graph products of finite groups. First results in this direction were obtained for automorphism groups of graph products of primary cyclic groups by Gutierrez, Piggott and Ruane  \cite[4.14]{Gutierrez}. Our theorem extends their result.

\begin{NewTheoremA}
Let $G_\Gamma$ be a graph product of finite groups. For ${\rm Aut}(G_\Gamma)$ the following statements are equivalent
\begin{enumerate}
\item[(i)] The group ${\rm Aut}(G_\Gamma)$ is word hyperbolic.
\item[(ii)] The group ${\rm Aut}(G_\Gamma)$ has no subgroup isomorphic to $\mathbb{Z}\times\mathbb{Z}$.
\item[(iii)] The graph $\Gamma$ has no induced cycle of length four and  does not contain any separating intersection of links SIL\footnote{ for definition see \ref{SIL}.}.
\end{enumerate} 
\end{NewTheoremA}

Another interesting property of groups is virtual freeness.
A group is called virtually free if it contains a free subgroup of finite index. Note that finite groups are virtually free. It is known that every finitely generated virtually free group is word hyperbolic. Virtually free graph products of finite groups were characterized in terms of $\Gamma$ by Lohrey and Senizergues in \cite{Lohrey}. They proved that a graph product of finite groups $G_\Gamma$ is virtually free iff $\Gamma$ is chordal (i. e. $\Gamma$ contains no induced cycles of length $\geq 4$). We investigate how to characterize virtual freeness of automorphism groups of graph products of finite groups and we obtain
\begin{NewPropositionB}
Let $G_{\Gamma}$ be a graph product of finite groups. For ${\rm Aut}(G_\Gamma)$ the following statements are equivalent
\begin{enumerate}
\item[(i)] The group ${\rm Aut}(G_\Gamma)$ is virtually free.
\item[(ii)] The graph $\Gamma$ is chordal and does not contain a SIL.
\end{enumerate} 
\end{NewPropositionB}

We are also interested in fixed point properties.
Recall that a group $G$ has Serre's fixed point property ${\rm F}\mathcal{A}$ if  every action of $G$ on a simplicial tree without inversions of edges has a global fixed point. It was proven by Dicks and Dunwoody in \cite[IV 1.9]{Dicks} that a finitely generated group $G$ is virtually free if and only if $G$ acts on a tree without inversions of edges and such that all vertex stabilizers are finite. Using this result and Proposition B we obtain

\begin{NewCorollaryC}
Let $G$ be a finitely generated group. If $G$ is infinite and virtually free, then $G$ does not have property ${\rm F}\mathcal{A}$. 

In partiular: Let $G_\Gamma$ be a graph product of finite groups. If $\Gamma$ is not complete, chordal and contains no SIL, then ${\rm Aut}(G_\Gamma)$ does not have Serre's fixed point property ${\rm F}\mathcal{A}$.
\end{NewCorollaryC}

Let us remark, that it was proven in \cite[Cor. B]{Varghese} that for $n\geq 4$ the group ${\rm Aut}(\mathbb{Z}_2^{*n})$ has property ${\rm F}\mathcal{A}$. This result was generalized in \cite[Thm. 1.1]{Leder} for groups ${\rm Aut}(\mathbb{Z}_{k}*\ldots*\mathbb{Z}_{k})$.
$ $\\

Another important class of groups which are also defined via simplicial graphs is the class of Coxeter groups. Given a finite simplicial graph $\Gamma=(V,E)$ with an edge-labeling $\varphi:E\rightarrow \mathbb{N}_{\geq 2}$. The corresponding Coxeter group is defined as follows $W_\Gamma=\langle V\mid v^2, (vw)^{\varphi(\left\{v, w\right\})}, v\in V, \left\{v, w\right\}\in E\rangle$.
Word hyperbolic Coxeter groups were characterized  in terms of $\Gamma$ by Moussong in \cite[17.1]{Moussong}. 
Our focus here is a study of word hyperbolicity for automorphism groups of Coxeter groups. We obtain the following results.
\begin{NewTheoremD}
Let $W_\Gamma$ be a Coxeter group. 
\begin{enumerate}
\item[(i)] If $W_\Gamma$ is not word hyperbolic, then ${\rm Aut}(W_\Gamma)$ is not word hyperbolic.
\item[(ii)] If $\Gamma$ has more than two connected components, then  ${\rm Aut}(W_\Gamma)$ is not word hyperbolic.
\item[(iii)] If $\Gamma$ has two connected components $\Gamma_1$ and $\Gamma _2$, then ${\rm Aut}(W_\Gamma)$ is word hyperbolic if and only if the parabolic subgroups $W_{\Gamma_1}$ and $W_{\Gamma_2}$ are finite.
\item[(iv)] If $W_\Gamma$ is word hyperbolic and ${\rm Out}(W_\Gamma)$ is finite, then ${\rm Aut}(G_\Gamma)$ is word hyperbolic.  
\end{enumerate} 
\end{NewTheoremD}

Virtually free Coxeter groups were described by Mihalik and Tschantz in \cite[Thm. 34]{Mihalik}. They showed, that  $W_\Gamma$ is virtually free iff $\Gamma$ is chordal and for every complete subgraph $\Gamma'\subseteq\Gamma$ the parabolic subgroup $W_{\Gamma'}$ is finite.
We want to answer the following question: For which shape of the graph $\Gamma$ is the automorphism group ${\rm Aut}(W_\Gamma)$ virtually free? We obtain the following partial result.
\begin{NewPropositionE}
Let $W_\Gamma$ be a Coxeter group. For ${\rm Aut}(W_\Gamma)$ the following statements are equivalent
\begin{enumerate}
\item[(i)] The group ${\rm Aut}(W_\Gamma)$ is virtually free 
\item[(ii)] The group ${\rm Out}(W_\Gamma)$ is finite, the graph $\Gamma$ is chordal and for every complete subgraph $\Gamma'\subseteq\Gamma$ the parabolic subgroup $W_{\Gamma'}$ is finite.
\end{enumerate}
\end{NewPropositionE} 

\section{Preliminaries} 
\subsection{Word hyperbolic groups}

In this subsection we define word hyperbolic groups and collect some important properties of these groups which we will need to prove our main theorems. A detailed description of word hyperbolic groups can be found in \cite{Bridson} and \cite{Gromov}.
We begin with a definition of a $\delta$-hyperbolic space.
\begin{definition}
Let $\delta\geq 0$. A geodesic triangle in a metric space is said to be $\delta$-slim if each of its sides is contained in the $\delta$-neighbourhood of the union of the other two sides. A geodesic space $X$ is said to be $\delta$-hyperbolic if every triangle in $X$ is $\delta$-slim. If $X$ is $\delta$-hyperbolic for some $\delta\geq 0$, we say that $X$ is word hyperbolic.
\end{definition}
A crucial property is that word hyperbolicity is stable under quasi-isometries \cite[III H.1.9]{Bridson}. Hence the following definition does not depend on the choice of a generating set.
\begin{definition}
A finitely generated group is called word hyperbolic (often abbreviated to hyperbolic) if its Cayley graph is $\delta$-hyperbolic metric space for some $\delta\geq 0$. 
\end{definition}
Classical examples of word hyperbolic
groups are finite groups and free groups. The standard example of a not word hyperbolic group is the direct product 
$\mathbb{Z}\times\mathbb{Z}$ of two infinite cyclic groups.

In order to prove that a given group $G$ is not word hyperbolic, it is enough to show that $G$ has a subgroup isomorphic to $\mathbb{Z}\times\mathbb{Z}$.
\begin{lemma}(\cite[III $\Gamma$.3.10]{Bridson})
\label{ZZ}
Let $G$ be a group. If $G$ is word hyperbolic, then $G$ does not contain a subgroup isomorphic to $\mathbb{Z}\times\mathbb{Z}$.
\end{lemma}
In some classes of groups this is the only obstruction. The next results follow from the fact that every subgroup $H\subseteq G$ of finite index is quasi-isometric to $G$ and two groups which differ by finite groups are also quasi-isometric \cite[p. 138]{OfficeHours}. 
\begin{lemma}
\label{quotient}
Let $G$ be a finitely generated group, let $H$ be a subgroup of finite index and $N$ be a finite normal subgroup.
\begin{enumerate}
\item[(i)] The group $H$ is word hyperbolic if and only if  $G$ is word hyperbolic. 
\item[(ii)] The group $G/N$ is word hyperbolic if and only if $G$ is word hyperbolic.
\end{enumerate}
\end{lemma}
Let $G$ be a group. We denote by $Z(G)$ the center of $G$, by ${\rm Inn}(G)$ the subgroup of ${\rm Aut}(G)$ consisting of inner automorphisms and by ${\rm Out}(G):={\rm Aut}(G)/{\rm Inn}(G))$ the outer automorphism group of $G$. 

The next result follows immediately from Lemma \ref{quotient}.
\begin{corollary}
\label{finiteOut}
Let $G$ be a word hyperbolic group. If $Z(G)$ and ${\rm Out}(G)$ are finite, then ${\rm Aut}(G)$ is word hyperbolic. 
\end{corollary}
\begin{proof}
By assumption the group $G$ is word hyperbolic and $Z(G)$ is finite.
Hence the word hyperbolicity of $G/Z(G)\cong{\rm Inn}(G)$ follows from Lemma \ref{quotient}(ii). The group ${\rm Inn}(G)$ has finite index in ${\rm Aut}(G)$, therefore it follows from Lemma  \ref{quotient}(i) that ${\rm Aut}(G)$ is word hyperbolic.
\end{proof}

\subsection{Virtually free groups}

In this subsection we introduce virtually free groups and we collect some facts about these groups which we will need to prove our results. 
\begin{definition}
A group $G$ is called virtually free if it contains a free subgroup of finite index. 
\end{definition}
It is obvious that finite groups and free groups are virtually free. The following result is a direct consequence of Lemma \ref{quotient}(i).
\begin{corollary}
\label{virhyp}
Let $G$ be a finitely generated group. If $G$ is virtually free, then $G$ is word hyperbolic.
\end{corollary}

We will also use the following property of virtually free groups which follows from \cite[IV 1.9]{Dicks}.
\begin{lemma}
\label{Dicks}
Let $G$ be a finitely generated virtually free group and $H\subseteq G$ be a subgroup. If $H$ is finitely generated, then $H$ is also virtually free. 
\end{lemma}

Being virtually free is stable under quasi-isometries. Thus we obtain the following
\begin{lemma}
\label{quotientVF}
Let $G$ be a finitely generated group and $N$ be a finite normal subgroup.
The group $G/N$ is virtually free if and only if $G$ is virtually free.
\end{lemma}

We end this section with the following result which is crucial for our proofs of Theorem B and Proposition E.
\begin{proposition}(\cite[3.1]{Pettet})
\label{Pettet}
Let $G$ be a finitely generated and virtually free group. Then ${\rm Aut}(G)$ is virtually free if and only if ${\rm Out}(G)$ is finite.
\end{proposition}

\section{Graph products of groups}
In this section we briefly present the main definitions and properties concerning graph products of groups. These groups are defined by presentations of a special form. A simplicial graph 
$\Gamma=(V,E)$ consists of a set $V$ of vertices and a set $E$ of two element subsets of $V$ which  are  called edges. If $\Gamma'=(V', E')$ is a subgraph of $\Gamma$ and $E'$ contains all the edges $\left\{v, w\right\}\in E$ with $v, w\in V'$, then $\Gamma'$ is called an induced subgraph of $\Gamma$. For $V'\subseteq V$ we denote by $\langle V'\rangle$ the smallest induced subgraph of $\Gamma$ with $(V',\emptyset)\subseteq \langle V'\rangle$. This subgraph is called a graph generated by $V'$. For a vertex $v\in V$ we define its link as ${\rm lk}(v)=\left\{w\in V\mid \left\{v,w\right\}\in E\right\}$ and its star as ${\rm st}(v)=\left\{v\right\}\cup {\rm lk}(v)$. A path of length $n$ is a graph $P_n=(V, E)$ of the form  $V=\left\{v_0, \ldots, v_n\right\}$ and $E=\left\{\left\{v_0, v_1\right\}, \left\{v_1, v_2\right\}, \ldots, \left\{v_{n-1}, v_n\right\}\right\}$ where the $v_i$, $0\leq i\leq n$, are pairwise distinct. If $P_n=(V,E)$ is a path of length $n\geq 3$, then the graph $C_{n+1}:=(V, E\cup\left\{\left\{v_n, v_0\right\}\right\})$ is called a cycle of length $n+1$. The distance between two vertices $u, v\in V$, denoted by $d(u,v)$, is the length of a shortest path from $u$ to $v$. A graph $\Gamma=(V,E)$ is called connected if any two vertices $v, w\in V$ are contained in a subgraph $\Gamma'$ of $\Gamma$ such that $\Gamma'$ is a path. A maximal connected subgraph of $\Gamma$ is called a connected component of $\Gamma$. A graph $\Gamma$ is called chordal if $\Gamma$ contains no induced cycles of length $\geq 4$.

The following definition is crucial for our results.
\begin{definition}
\label{SIL}
A graph $\Gamma=(V, E)$ has a separating intersection of links (abbreviated SIL) if there exist two vertices $v, w\in V$ with $d(v,w)\geq 2$ and there is a connected component of $\langle V-({\rm lk}(v)\cap{\rm lk}(w))\rangle$ which contains neither $v$ nor $w$.
\end{definition}

A graph with no SILs is either connected or it is the disjoint union of two complete graphs, see \cite[3.3]{Charney}.
Further, if $\Gamma$ is a tree, then $\Gamma$ contains no SIL if and only if the valence of every vertex is less than $3$.

\begin{definition}
Let $\Gamma=(V, E)$ be a finite simplicial graph and
 $\left\{G_v\mid v\in V\right\}$ be a set of non-trivial groups. The graph product $G_\Gamma$ is defined as the quotient
\[
({\Large{*}}_{v\in V}G_v)/ \langle\langle gh=hg\mid  g\in G_v, h\in G_w, \left\{ v, w\right\}\in E\rangle\rangle.
\] 
\end{definition}
Let us consider some examples. If $\Gamma$ is a discrete graph, then $G_\Gamma$ is a free product of $G_v$, $v\in V$ and if $\Gamma$ is a complete graph, then $G_\Gamma$ is a direct product of $G_v, v\in V$. 
Further, if all $G_v$ are infinite cyclic, then $G_\Gamma$ is called a right angled Artin group and if all $G_v$ have order two, then $G_\Gamma$ is known as a right angled Coxeter group.

If $\Gamma$ is not complete, then $G_\Gamma$ is always infinite. This follows from the following result.
\begin{lemma}(\cite{Green})
Let $G_\Gamma$ be a graph product and $\Gamma'\subseteq\Gamma$ a  subgraph. If $\Gamma'$ is an induced subgraph, then $G_{\Gamma'}$ is a subgroup of $G_\Gamma$ and is called a parabolic subgroup.
\end{lemma}

In order to prove that many automorphism groups of graph products  have subgroups isomorphic to $\mathbb{Z}\times\mathbb{Z}$ we need a precise definition of some elements of ${\rm Aut}(G_\Gamma)$.
\begin{definition}
Let $G_\Gamma$ be a graph product, $v\in V$, $x\in G_v$ and $C=(V', E')$ be a connected component of a subgraph generated by $V-{\rm st}(v)$. The partial conjugation $\pi_{x, C}$ in an automorphism of $G_\Gamma$  induced by: 
\[h\mapsto \left\{
\begin{array}{ll}
xhx^{-1} & h\in G_w, w\in V' \\
h & h\in G_w, w\notin V'.  \\
\end{array}
\right. \]
\end{definition}
Now we are able to prove 

\begin{lemma}
\label{subgroupZZ}
Let $G_\Gamma$ be a graph product. If $\Gamma$ contains a SIL, then ${\rm Aut}(G_\Gamma)$ has a subgroup isomorphic to $\mathbb{Z}\times\mathbb{Z}$.
\end{lemma}
\begin{proof}
If $\Gamma$ contains a SIL, then there exist vertices $v, w\in V$ with $d(v, w)\geq 2$ and a connected component $C$ of $\langle V-{\rm lk}(v)\cap{\rm lk}(w)\rangle$ which contains neither $v$ nor $w$. Let $x\in G_v$ and $y\in G_w$ be non-trivial elements and $f$ denote the inner automorphism $g\mapsto xygy^{-1}x^{-1}$. It is obvious that the order of $f$ in infinite. Further the automorphism $\pi_{x, C}\circ \pi_{y, C}$ also has infinite order and $\langle f, \pi_{x, C}\circ\pi_{y, C}\rangle\cong\mathbb{Z}\times\mathbb{Z}$. 
\end{proof}

\begin{proposition}
\label{freeProducts}
Let $G_\Gamma$ be a graph product. 
\begin{enumerate}
\item[(i)] If $\Gamma=(V, \emptyset)$ and $\#V\geq 3$, then ${\rm Aut}(G_\Gamma)$ is not word hyperbolic.
\item[(ii)] If $\Gamma=(\left\{v, w\right\}, \emptyset)$ and the vertex groups $G_{v}$ and $G_{w}$ have finite center, then ${\rm Aut}(G_\Gamma)$ is word hyperbolic if and only if  $G_{v}$ and $G_{w}$ are finite.
\end{enumerate} 
\end{proposition}
\begin{proof}
The first statement of the proposition follows from Lemma \ref{subgroupZZ} and Lemma \ref{ZZ}. 

Let $\Gamma=(\left\{v, w\right\}, \emptyset)$ and let $Z(G_v)$ and $Z(G_w)$ be finite. If $G_v$ and $G_w$ are finite, then it was proven in \cite{Karrass} that ${\rm Aut}(G_v*G_w)$ is an amalgamated product of finite groups. The class of word hyperbolic groups is closed under free products with amalgamation along finite subgroups \cite[Cor. 3]{Kharlampovich}. Hence ${\rm Aut}(G_v*G_w)$ is word hyperbolic.

Suppose that $G_v$ is not finite. If $G_v$ is infinite torsion, then the group ${\rm Aut}(G_v*G_w)$ has an infinite torsion subgroup. More precisely,  $G_v*G_w\cong {\rm Inn}(G_v*G_w)\subseteq{\rm Aut}(G_v*G_w)$. It was proven by Gromov that torsion subgroups of word hyperbolic groups are finite, see \cite[Chap.8 Cor.36]{Ghys}. Thus ${\rm Aut}(G_v*G_w)$ is not word hyperbolic. 

Otherwise there exists an element of infinite order in $G_v$. Let $x\in G_v$ be an element with infinite order and let $f$ denote the inner automorphism $g\mapsto xgx^{-1}$. Notice that $f\neq\pi_{x,\langle\left\{w\right\}\rangle}$ since $Z(G_v)$ is finite. Then $\langle f, \pi_{x,\langle\left\{w\right\}\rangle}\rangle\cong\mathbb{Z}\times\mathbb{Z}$. It follows again from Lemma \ref{ZZ} that ${\rm Aut}(G_\Gamma)$ is not word hyperbolic.
\end{proof}

In this paper our focus is on graph products of finite groups, that is, graph products for which all of the vertex groups $G_v$ are finite. Word hyperbolic graph products of finite groups were characterized by Meier in \cite{Meier}. 
\begin{theorem}(\cite{Meier})
\label{hyperbolicG}
Let $G_\Gamma$ be a graph product of finite groups. For $G_\Gamma$ the following statements are equivalent
\begin{enumerate}
\item[(i)] The group $G_\Gamma$ is word hyperbolic.
\item[(ii)] The group $G_\Gamma$ has no subgroup isomorphic to $\mathbb{Z}\times\mathbb{Z}$.
\item[(iii)] The graph $\Gamma$ does not contain an induced cycle of length four.
\end{enumerate}
\end{theorem}
Virtually free graph products of finite groups were described by Lohrey and Senizergues in \cite[Thm. 1.1]{Lohrey}.

\begin{theorem}(\cite[Thm. 1.1]{Lohrey})
\label{virtuallyfreeG}
Let $G_\Gamma$ be a graph product of finite groups. For $G_\Gamma$ the following statements are equivalent
\begin{enumerate}
\item[(i)] The group $G_\Gamma$ is virtually free.
\item[(ii)] The graph $\Gamma$ is chordal.
\end{enumerate}
\end{theorem}

The last result which we will need to prove Theorem A and Proposition B is the following:
\begin{proposition}(\cite[3.20]{Genevois})
\label{finOut}
Let $G_\Gamma$ be a graph product of finite groups. Then ${\rm Out}(G_\Gamma)$ is finite if and only if $\Gamma$ contains no SIL.  
\end{proposition}

\section{Proof of Theorem A}
Now we have all the ingredients to prove Theorem A. 
\begin{proof}
Let $G_\Gamma$ be a graph product of finite groups. 
It is straightforward to verify that the center of $G_\Gamma$ is contained in a parabolic subgroup $G_{\langle V'\rangle}$ where 
\[
V'=\left\{ v\in V\mid d(v, w)\leq 1\text{ for each }w\in V\right\}.
\]
The induced subgraph $\langle V'\rangle$ is complete. Thus $G_{\langle V'\rangle}$ is finite and therefore $Z(G_\Gamma)$ is also finite.

Assume that ${\rm Aut}(G_\Gamma)$ is word hyperbolic. Then by Lemma \ref{ZZ} ${\rm Aut}(G_\Gamma)$ has no subgroup isomorphic to $\mathbb{Z}\times\mathbb{Z}$. 

Suppose that ${\rm Aut}(G_\Gamma)$ has no subgroup isomorphic to $\mathbb{Z}\times\mathbb{Z}$. We have ${\rm Inn}(G_\Gamma)\cong G_\Gamma/Z(G_\Gamma)$. Since $Z(G_\Gamma)$ is finite, $G_\Gamma$ has also no subgroup isomorphic to $\mathbb{Z}\times\mathbb{Z}$. By Theorem \ref{hyperbolicG} follows that $\Gamma$ has no induced cycle of length $4$. Let us assume that $\Gamma$ has a SIL, then 
by Lemma \ref{subgroupZZ} the group ${\rm Aut}(G_\Gamma)$ has a subgroup isomorphic to $\mathbb{Z}\times\mathbb{Z}$. Hence ${\rm Aut}(G_\Gamma)$ is not hyperbolic. This contradicts (ii). Therefore $\Gamma$ contains no SIL. 

Assume now that $\Gamma$ has no induced cycle of length four and $\Gamma$ does not contain a SIL. By Proposition \ref{finOut} the outer automorphism group of $G_\Gamma$, ${\rm Out}(G_\Gamma)$ is finite. Further, the center of $G_\Gamma$ is finite and by Theorem \ref{hyperbolicG} the group $G_\Gamma$ is word hyperbolic. The word hyperbolicity of ${\rm Aut}(G_\Gamma)$ follows  from Corollary \ref{finiteOut}.
\end{proof}

\section{Coxeter groups} 
An another important class of groups is a class consisting of Coxeter groups. 
\begin{definition}
Let $\Gamma=(V, E)$ be a finite simplicial graph with an edge-labeling $\varphi:E\rightarrow \mathbb{N}_{\geq 2}$. The Coxeter group $W_\Gamma$ is defined as follows:
\[
W_\Gamma=\langle V\mid v^2, (vw)^{\varphi(\left\{v, w\right\})}, v\in V, \left\{v, w\right\}\in E\rangle.
\]
\end{definition}
For example, the symmetric group ${\rm Sym}(n)$ and the infinite dihedral group $D_\infty$ are Coxeter groups. We collect some of the facts about Coxeter groups that we need. For more information about these groups we refer to \cite{Davis} and \cite{Humphreys}.
\begin{lemma}(\cite[4.1.6]{Davis})
Let $W_\Gamma$ be a Coxeter group and $\Gamma'\subseteq\Gamma$ be a  subgraph. If $\Gamma'$ is an induced subgraph, then $W_{\Gamma'}$ is a subgroup of $W_\Gamma$ and is called a parabolic subgroup.
\end{lemma}

\begin{lemma}(\cite[1.1]{Hosaka})
\label{center}
Let $W_\Gamma$ be a Coxeter group. Then $Z(W_\Gamma)$ is finite.
\end{lemma}

Word hyperbolic Coxeter groups were characterized by Moussong in his thesis \cite[17.1]{Moussong}. He proved
\begin{theorem}
\label{Moussong}
Let $W_\Gamma$ be a Coxeter group. Then the following are equivalent:
\begin{enumerate}
\item[(i)] The group $W_\Gamma$ is word hyperbolic.
\item[(ii)] The group $W_\Gamma$ has no subgroup isomorphic to $\mathbb{Z}\times\mathbb{Z}$.
\item[(iii)] There is no parabolic subgroup $W_{\Gamma'}$ such that $W_{\Gamma'}$ is an Euclidean Coxeter group of rank $\geq 3$ and there is no pair of disjoint induced subgraphs $\Gamma_1$ and $\Gamma_2$ such that the parabolic subgroups $W_{\Gamma_1}$ and $W_{\Gamma_2}$ commute and are infinite.
\end{enumerate}
\end{theorem}

Coxeter groups are fundamental, ’well understood’ objects of group theory, but there are many open questions concerning their  automorphism groups. In general, it is not known if the automorphism group of an arbitrary Coxeter group is finitely
presented.  We are interested in the word hyperbolicity of the automorphism groups of Coxeter groups.

First of all, we want to show that the hyperbolicity of the Coxeter group is a necessary condition for the hyperbolicity of the automorsphism group.
\begin{proposition}
\label{hyp}
Let $W_\Gamma$ be a Coxeter group. If $W_\Gamma$ is not word hyperbolic, then ${\rm Aut}(W_\Gamma)$ is not word hyperbolic.
\end{proposition}
\begin{proof}
Let $W_\Gamma$ be a not word hyperbolic Coxeter group. By Theorem \ref{Moussong} there exists a subgroup $H\subseteq W_\Gamma$ such that $H\cong\mathbb{Z}\times\mathbb{Z}$. Let us consider the canonical projection $\pi:W\rightarrow W_\Gamma/Z(W_\Gamma)$. By Lemma \ref{center} the center  $Z(W_\Gamma)$ is finite, therefore the restriction $\pi_{|H}$ is injective and thus  $\mathbb{Z}\times\mathbb{Z}\cong \pi(H)\subseteq W_\Gamma/Z(W_\Gamma)\cong{\rm Inn}(W_\Gamma)\subseteq{\rm Aut}(W_\Gamma)$. It follows from Lemma \ref{ZZ}  that ${\rm Aut}(W_\Gamma)$ is not word hyperbolic.
\end{proof}

Virtually free Coxeter groups were characterized in terms of graphs by Mihalik and Tschantz in \cite[Thm. 34]{Mihalik}. 

\begin{theorem}
Let $W_\Gamma$ be a Coxeter group. Then the following are equivalent:
\begin{enumerate}
\item[(i)] The group $W_\Gamma$ is virtually free.
\item[(ii)] The graph $\Gamma$ is chordal and for every complete subgraph $\Gamma'\subseteq\Gamma$ the parabolic subgroup $W_{\Gamma'}$ is finite.
\end{enumerate}
\end{theorem}

\section{Proof of Theorem D}
We almost proved the results of Theorem D. Let us recall the arguments.
\begin{proof}
The first statement of Theorem B was proven in Proposition \ref{hyp}. 

The second statement of Theorem B follows from Proposition \ref{freeProducts}(i). More precisely, let $\Gamma_1, \ldots, \Gamma_n$ be the connected components of $\Gamma$. Then $W_\Gamma$ is a free product of the parabolic subgroups $W_{\Gamma_i}$ and we consider $W_\Gamma$ as a graph product where the graph has $n$ vertices, no edges  and the vertex groups are $W_{\Gamma_i}$. If $n\geq 3$, then ${\rm Aut}(W_\Gamma)$ is not word hyperbolic by Proposition \ref{hyp}(i). 

The third statement of Theorem B follows from Proposition \ref{freeProducts}(ii). The graph $\Gamma$ has by assumption two connected components $\Gamma_1$ and $\Gamma_2$. We have $W_\Gamma=W_{\Gamma_1}*W_{\Gamma_2}$. The parabolic subgroups $W_{\Gamma_1}, W_{\Gamma_2}$ have by Lemma \ref{center} finite center. It follows from Proposition \ref{hyp}(ii) that ${\rm Aut}(W_\Gamma)$ is word hyperbolic iff  the parabolic subgroups $W_{\Gamma_1}$ and $W_{\Gamma_2}$ are finite. 

Using the fact that Coxeter groups always have  finite center and Corollary \ref{finiteOut} we obtain the result of Theorem B(iv). 
\end{proof}

\section{Hyperbolic groups with fixed point property ${\rm F}\mathcal{R}$}

Let us start with the following definition.
\begin{definition}
Let $\mathcal{X}$ be a class of metric spaces. A group $G$ is said to have property ${\rm F}\mathcal{X}$ if any action of $G$ by isometries on any member of $\mathcal{X}$ has a fixed point.
\end{definition}
In this section we concentrate on actions on metric trees.
We denote by $\mathcal{R}$ the class of metric trees and by $\mathcal{A}$ the class of simplicial
trees. The study of groups acting on simplicial trees was initiated by Serre \cite{Serre}. A slight generalization of Serre's fixed point property ${\rm F}\mathcal{A}$ is the fixed point property ${\rm F}\mathcal{R}$. 

One main technique to prove property ${\rm F}\mathcal{R}$ for groups is the following result.
\begin{proposition}
\label{finite}
Let $G$ be a group and $S$ be a finite generating set of $G$. If each $2$-element subset of $S$ generates a finite subgroup, then $G$ has property ${\rm F}\mathcal{R}$.
\end{proposition}
\begin{proof}
This follows immediately from \cite[Lemma 1]{BridsonFR} or \cite[2.2.2]{VarghesePhD} and the fact that finite groups always have property ${\rm F}\mathcal{R}$, see  \cite[Main Result 1]{Struyve} or \cite[2.2.3]{VarghesePhD}.
\end{proof}

The next proposition characterized graph products of finite groups and Coxeter groups with property ${\rm F}\mathcal{R}$.
\begin{proposition}
\label{CoxeterFR}
$ $
\begin{enumerate}
\item[(i)] Let $G_\Gamma$ be a graph product of finite groups. Then $G_\Gamma$ has property ${\rm F}\mathcal{R}$ iff $\Gamma$ is a complete graph.
\item[(ii)] Let $W_\Gamma$ be a Coxeter group. Then $W_\Gamma$ has property ${\rm F}\mathcal{R}$ iff $\Gamma$ is a complete graph.
\end{enumerate}
\end{proposition}
\begin{proof}
Let $G_\Gamma$ be a graph product of finite groups. If $\Gamma$ is a complete graph, then $G_\Gamma$ is a direct product of finite groups and thus finite. By \cite[2.2.3]{VarghesePhD} finite groups always have property ${\rm F}\mathcal{R}$. If $\Gamma$ is not complete, then there exists $v, w\in V$ such that $\left\{v, w\right\}\notin E$. 
Thus $G_\Gamma=G_{\langle V-\left\{v\right\}\rangle}*_{G_{\langle{\rm lk}(v)\rangle}}G_{\langle{\rm st}(v)\rangle}$ and by \cite[6.1 Thm. 15]{Serre} follows that $G_\Gamma$ does not have property ${\rm F}\mathcal{R}$.

Let $W_\Gamma$ be a Coxeter group. If $\Gamma$ is a complete graph, then by Proposition \ref{finite} follows that $W_\Gamma$ has property ${\rm F}\mathcal{R}$. If $\Gamma$ is not complete, then there exists $v, w\in V$ such that $\left\{v, w\right\}\notin E$. 
Thus $W_\Gamma=W_{\langle V-\left\{v\right\}\rangle}*_{W_{\langle{\rm lk}(v)\rangle}}W_{\langle{\rm st}(v)\rangle}$ (see \cite[Lemma 25]{Mihalik}) and by \cite[6.1 Thm. 15]{Serre} follows that $W_\Gamma$ does not have property ${\rm F}\mathcal{R}$.
\end{proof}

The following connection between the finiteness of the outer automorphism group of a hyperbolic group and property ${\rm F}\mathcal{R}$ was proven in \cite{Paulin}.
\begin{theorem}
\label{Out}
Let $G$ be a word hyperbolic group. If $G$ has property ${\rm F}\mathcal{R}$, then ${\rm Out}(G)$ is finite.
\end{theorem}

Regarding word hyperbolicity of automorphism groups of Coxeter groups we obtain the following result.
\begin{corollary}
Let $W_\Gamma$ be a word hyperbolic Coxeter group. If $\Gamma$ is a complete graph, then ${\rm Aut}(W_\Gamma)$ is word hyperbolic. 
\end{corollary}
\begin{proof}
Let $\Gamma$ be a complete graph and $W_\Gamma$ be a word hyperbolic Coxeter group. By Proposition \ref{CoxeterFR}(ii) the group $W_\Gamma$ has property ${\rm F}\mathcal{R}$. Using Theorem \ref{Out}  we obtain the finiteness of ${\rm Out}(W_\Gamma)$. The word hyperbolicity of ${\rm Aut}(W_\Gamma)$ follows from Theorem B(iv).
\end{proof}

We want to remark that it was proven in \cite[1.1]{Howlett} that the outer automorphism group of an arbitrary Coxeter group $W_\Gamma$ where $\Gamma$ is a complete graph is always finite.

\section{Proof of Proposition B}
\begin{proof}
We suppose first that ${\rm Aut}(G_\Gamma)$ is virtually free. The group ${\rm Inn}(G_\Gamma)\cong G_\Gamma/Z(G_\Gamma)$ is finitely generated and is by Lemma \ref{Dicks} virtually free. Using the finiteness of $Z(G_\Gamma)$ and Lemma \ref{quotientVF} we obtain virtual freeness of $G_\Gamma$. It follows from Theorem \ref{virtuallyfreeG} that $\Gamma$ is chordal. Assume now that $\Gamma$ contains a SIL. Then by Lemma \ref{subgroupZZ} the group ${\rm Aut}(G_\Gamma)$ has a subgroup isomorphic to $\mathbb{Z}\times\mathbb{Z}$. Thus ${\rm Aut}(G_\Gamma)$ is not word hyperbolic and by Corollary \ref{virhyp} the group ${\rm Aut}(G_\Gamma)$ is not virtually free. This completes one direction of the proof.

We suppose now that $\Gamma$ is chordal and does not contain a SIL. By Theorem \ref{virtuallyfreeG} the group $G_\Gamma$ is virtually free and by Proposition \ref{finOut} the group ${\rm Out}(G_\Gamma)$ is finite. Now we apply the result of Proposition \ref{Pettet} and this shows that ${\rm Aut}(G_\Gamma)$ is virtually free.
\end{proof}

The structure of the proof of Proposition E is the same as above.

\end{document}